\documentclass{article}
\usepackage{fullpage}
\usepackage{biolinum}
\usepackage[justification=centering]{caption}
\usepackage{xurl}
\usepackage[defaultlines=2,all]{nowidow}
\usepackage{nicefrac}
\usepackage{tikz,pgf}
\usetikzlibrary{calc}
\usepackage{mathrsfs}
\usepackage{enumitem}
\usepackage{amssymb,amsmath,amsthm}
\usepackage{mathtools}
\theoremstyle{definition}

\newtheorem{construction}{Construction}[section]

\newtheorem{example}[construction]{Example}

\theoremstyle{plain}

\newtheorem{theorem}[construction]{Theorem}
\newtheorem{corollary}[construction]{Corollary}
\newtheorem{lemma}[construction]{Lemma}

\theoremstyle{remark}

\newtheorem*{remark*}{Remark}
\newcommand*{\R}[1]{\textcolor{red}{#1}}
\newcommand*{\ceil}[1]{\lceil #1 \rceil}
\title{On min-base palindromic representations of powers of 2}
\author{%
Donald L. Kreher%
\thanks{Department of Mathematical Sciences,
Michigan Technological University,
Houghton, MI 49931-1295, U.S.A.}
\and
Douglas R. Stinson%
\thanks{David R.~Cheriton School of Computer Science, University of Waterloo,  Waterloo ON, N2L 3G1, Canada}
}
\begin{document}
\maketitle
\begin{abstract}
A positive integer  $N$ is \emph{palindromic in the base $b$}  when
$N = \sum_{i=0}^{k} c_i b^i$,
$c_k\neq 0$,and
$c_i=c_{k-i},\; i=0,1,2,...,k$, 
Focusing on powers of 2, we investigate the smallest base $b$
when $N=2^n$ is palindromic in the base $b$.
\end{abstract}
\section{Introduction}
Let $b > 1$ be an integer. Then every positive integer
$N$ can be written uniquely in the form
\[
N = \sum_{i=0}^k c_i b^i,kk
\]
where $k \geq 0$ and $0\leq c_i <b$ are integers with $c_k >0$. We
write 
\[
N = (c_k,c_{k-1},c_{k-2},\cdots ,c_0)_{b}
\]
and say $N$ has representation 
$(c_k,c_{k-1},c_{k-2},\ldots, c_0)$ 
in \emph{radix} or \emph{base} $b$. 
For example,
\[2023 = (2,0,2,3)_{10} = (5,6,2,0)_{7} = (7,14,7)_{16}=(3,18,22)_{23}.\]
The coefficients
$c_k, c_{k-1}, \ldots, c_0$ are called the \emph{digits} of $N$
in base $b$ and $c_k$ is the \emph{leading digit}. 

If $N = (c_k,c_{k-1},c_{k-2},\cdots, c_0)_{b}$ and
$c_j = c_{k-j}$ for all $j=0,1,2,\ldots,k$, then
we say $N$ is a \emph{palindrome in the base $b$} and that 
$(c_k,c_{k-1},c_{k-2},\cdots, c_0)$ is a 
\emph{$(k+1)$-digit palindromic representation} of $N$ in the base $b$.
For example, $2023$ has a 3-digit palindromic representation in the base 16.   
\begin{remark*}
A palindromic representation that contains leading zeros can 
be reduced. If 
\[
N = (%
\underbrace{0,0,\ldots,0}_{\text{$z$ digits} },
c_{k-z},c_{k-z-1},c_{k-z-2},\cdots, c_z,
\underbrace{0,0,\ldots,0}_{\text{$z$ digits} }
)_{b}
\]
then $N = b^{z-1}M$,
where $M = (c_{k-z},c_{k-z-1},c_{k-z-2},\ldots, c_{z})_b$ is also a palindromic representation.
Thus we only consider palindromic representation with a
nonzero leading digit.
\end{remark*}

There are  not many papers that discuss palindromic representations in different bases. Indeed,  we were able to find only three published results, but there may be others.
\begin{enumerate}[label=(\roman{*})]
\item In~\cite{CLB},  the authors prove, for every base $b \geq 5$,
that any positive integer can be written
as a sum of three palindromes in base $b$.

\item In~\cite{G}, 
the author proves that there exist exactly 203 positive integers $N$ such that $N$ is a palindrome in base 10 with $d \geq 2$ digits and $N$ is also a palindrome with $d$ digits in a base $b\neq 10$. 

\item In~\cite{PP},
for fixed base $b$,
the  authors investigate the number of positive integers up to $n$
that are  palindromic in  base $b$. 

\end{enumerate}
The only discussion pertaining to the problem that we investigate is the website~\cite{Brown}. 

It is not difficult to see that $N = (1,1)_{N-1}$. Hence
every positive integer has a palindromic representation.
We define $b=b(N)$ to be the smallest base $b>1$ such that
$N$ has a palindromic representation in the base $b$. 
Table~\ref{b(N)} enumerates $b(N)$ for all $N \leq 100$.
The red entries in Table~\ref{b(N)} are where $b(N)=N-1$.

\begin{table}[tb]
\begin{center}
\caption{$b(N)$, for $N=i+j$, 
where $i=0,10,20,\ldots,90$
and  $j=0,1,2,\ldots,9$.}
\label{b(N)}
\medskip
\begin{tabular}{l|*{10}{r}}
\tikz[baseline=(O),scale=1,x=0.5em,y=0.5em]
{
 \draw (-1,1)--(1,-1);
 \coordinate (O) at (0,0);
 \node[right=7pt,above=-5pt] at (O) {$i$};
 \node[left =7pt,below=-5pt] at (O) {$j$};
}
   &  0&  1&  2&  3&  4&  5&  6&  7&  8&  9\\\hline
  0&   &  2&  3&\R{  2}&\R{  3}&  2&\R{  5}&  2&  3&  2\\
 10&  3&\R{ 10}&  5&  3&  6&  2&  3&  2&  5&\R{ 18}\\
 20&  3&  2& 10&  3&  5&  4&  3&  2&  3&  4\\
 30&  9&  2&  7&  2&  4&  6&  5&  6&  4& 12\\
 40&  3&  5&  4&  6& 10&  2&  4&\R{ 46}&  7&  6\\
 50&  7&  2&  3&\R{ 52}&  8&  4&  3&  5& 28&  4\\
 60&  9&  6&  5&  2&  7&  2& 10&  5&  3& 22\\
 70&  9&  7&  5&  2&  6& 14& 18& 10&  5&\R{ 78}\\
 80&  3&  8&  3&  5& 11&  2&  6& 28&  5&  8\\
 90& 14&  3&  6&  2& 46& 18& 11&  8&  5&  2\\
\end{tabular}
\end{center}
\end{table}

\vbox{
\begin{theorem}[K. Brown~\cite{Brown}]  
If  $b(N) = N-1$,
then $N=3,4,6$ or $N>6$ and is a prime.

\begin{proof}
It is easy to check that $b(N) \neq N-1$ when $N = 1,2,5$. Also, 
$b(N)=N-1$ when $N=3,4,6$. Hence, we suppose $N > 6$.

Suppose  $N=ab$ with $a < b-1$, 
where $b > 2$. Then $N=a(b-1)+a$, 
so $N$ has the palindromic representation  $(a,a)_{b-1}$.

This covers all composites greater than 6 except for squared primes.
For squares, we have \[a^2 = (a-1)^2 + 2(a-1) + 1,\] so every square 
$N=a^2 > 4$ has the palindromic representation $(1,2,1)_{a-1}$.  
\end{proof}

\end{theorem}
}
\begin{remark*}
Observe that $b(13)=3$, because $13=(1,1,1)_{3}$, so
not every prime $N$ has $b(N)=N-1$.
\end{remark*}

If 
$N = (\alpha c_k,\alpha c_{k-1},\cdots \alpha c_0)_{b}$,
then 
$N= \alpha(c_k, c_{k-1}, \cdots  c_0)_{b}= \alpha M$, 
where 
$M= (c_k, c_{k-1}, \cdots  c_0)_{b}$ and we say that
the representation of $N$ in the base $b$
is a \emph{multiple} of the representation of $M$ in the base $b$.
Conversely if 
$N = \alpha M$ 
and
$M= (c_k, c_{k-1}, \cdots  c_0)_{b}$, 
where 
$\alpha c_i < b$,
for 
$i=0,1,2,\ldots,k$, then
$N= \alpha(c_k, c_{k-1}, \cdots  c_0)_{b}$. 

For example, we have
\begin{align}
2023 
&= (7,14,7)_{16} 
= 7(1,2,1)_{16}
= 7 \cdot 289.
\end{align}
The representation of $2023$ is a multiple of the representation of $289$.

We say that the representation $(\alpha c_k,\alpha c_{k-1},\cdots \alpha c_0)$
of $N$ in the base $b$ has \emph{binomial form} or 
is a \emph{multiple $\alpha$ of a  binomial representation}
if $c_i= \binom{k}{i}$, $i=0,1,2,\ldots,k$. 
Every  multiple of a  binomial representation in the base $b$
is of course a palindromic representation. 

The following easy results can be deduced from~\cite{Brown}.
\begin{lemma}\label{binom}
If $\alpha \dbinom{k}{\lceil k/2 \rceil} < b$ and $N=\alpha (1+b)^k$,
then $N$ has binomial form
\[
\alpha \Bigg( 
\binom{k}{k},
\binom{k}{k-1},
\binom{k}{k-2},
\ldots,
\binom{k}{0}
\Bigg)_{b}.
\]
\end{lemma}
\begin{corollary} If $n=xk+r$ with $r \geq 0$ and $2^r \binom{k}{\lceil k/2 \rceil} < b$, then $2^n$ has a 
representation in base $b$ in {binomial form}. 

\begin{proof}
If $n=xk+r$, then, setting $b=2^x-1$ and $\alpha=2^r$, we 
obtain $N=2^n =\alpha(1+b)^k$. Because $2^r \binom{k}{\lceil k/2 \rceil} < b$, 
Lemma~\ref{binom} applies 
and we have a representation of $2^n$ in base 
$2^x-1$ in binomial form.
\end{proof}
\end{corollary}

Kevin Brown writes in his investigation~\cite{Brown} 
\begin{quotation}
\noindent
``This raises the question of whether the min-base
representation for powers of 2 is always of the binomial form.''
\end{quotation}
In the rest of this paper, we prove some partial results regarding this question and we also report some computational results.

\section{New Results}

\begin{lemma}
\label{bin.lem}
If the representation of $2^n$ in the base $b$ is a multiple $\alpha$
of a binomial representation, then $b = 2^x-1$ for some integer $x \leq n$ and $\alpha$ is a power of $2$.
\end{lemma}
\begin{proof}
We have 
\[2^n = \sum_{i=0}^k \alpha \binom{k}{i} b^i\]
for some positive integer $\alpha$. 
However, \[\sum_{i=0}^k \alpha \binom{k}{i} b^i = \alpha(b+1)^k.\]
Hence, $b+1$ is a divisor of $2^n$ and thus $b = 2^x-1$ for some $x \leq n$. 
Then
\[2^n = \alpha \, 2^{kx},\] so $\alpha = 2^{n-kx}$ and therefore $\alpha$ is a power of $2$.
\end{proof}

\begin{lemma}\label{Marco}
If $N=(c_{2k+1},c_{2k},\ldots,c_1,c_0)_{b}$
is a palindromic representation,
then $b+1$ divides $N$.
\end{lemma}
\begin{proof}
Because $c_i=c_{2k+1-i}$, $i=0,1,2,\ldots,k$, we have
\[
N=\sum_{i=0}^{2k+1} c_i b^{i}
=\sum_{i=0}^{k} c_i (b^{2k+1-i}+ b^{i})
=\sum_{i=0}^{k} c_i (b^{2k+1-i}+ b^{i})
\]
Then, because $b+1$ divides
$b^{2k+1-i}+ b^{i}$ for $i=0,1,\ldots,k$, we have
that $b+1$ divides  $N$. 
\end{proof}

\begin{corollary}\label{Leonard-2}
If $N=2^n$ has a palindromic representation in base $b$ 
with an even number of digits,
then $b=2^x-1$ for some $x$.
\end{corollary}
\begin{proof}
Applying Lemma~\ref{Marco}, we see that $b+1$ divides $2^n$.
Therefore $b=2^x-1$, for some $x$.
\end{proof}

We now show that Corollary~\ref{Leonard-2} does not necessarily hold if $k$ is even. We consider palindromic representations of $2^n$ with $k = 2$ (i.e., palindromic representations with three digits),
which necessarily have the form $2^n = (c,d,c)_b$. We present examples of such representations where $b$ is not of the form
$2^x-1$ for some integer $x$.

\begin{theorem}
\label{k=2}
$2^n$ has a 3-digit palindromic representation in the base $b$ if and only if
$b^2 + 1 \leq 2^n \leq b^3-1$ and 
\begin{equation}
\label{3digit.eq} \frac{1}{b}\left\lfloor \frac{2^n}{b} \right\rfloor- \frac{b-1}{b} \leq c \leq \frac{1}{b}\left\lfloor \frac{2^n}{b} \right\rfloor,
\end{equation}
where $c = 2^n \bmod b \neq 0$.
\end{theorem}
\begin{proof}
Suppose that 
\begin{equation}
\label{k2eq}
2^n = c + db + cb^2
\end{equation}
with $1 \leq c \leq b-1$ and $0 \leq d \leq b-1$.
The smallest palindromic representation with 3 digits is $(1,0,1)_b$ and the largest is $(b-1,b-1,b-1)_b$, 
so we must have \[b^2 + 1 \leq 2^n \leq b^3-1.\]
Reducing equation (\ref{k2eq}) modulo $b$, we have $2^n \equiv c \pmod{b}$. Since $0 \leq c \leq b-1$, we have $c = 2^n \bmod b$.

Next, we can express $2^n = kb + c$, where $k = \lfloor \frac{2^n}{b} \rfloor$.
We have 
\[ c(1 + b^2) + db = 2^n = kb + c,\]
from which it follows that $d = k - cb$.
We require $0 \leq d = k - cb \leq b-1$, so 
\[ c \leq \frac{k}{b} = \frac{1}{b}\left\lfloor \frac{2^n}{b} \right\rfloor \]
and
\[ c \geq \frac{k-b+1}{b} = \frac{1}{b}\left\lfloor \frac{2^n}{b} \right\rfloor- \frac{b-1}{b}. \]
These necessary conditions are also sufficient, so the desired result follows. 
\end{proof}

\begin{theorem}
\label{k=2 BNF}
The only 3-digit binomial form representations of $2^n$ are
\[
(2^i)(1,2,1)_{2^{(n-i)/2}-1},
\]
where $0 \leq  3i < n-2$ and $n \equiv i \pmod{2}$.
\end{theorem}
\begin{proof}
If $2^n$ has a 3-digit binomial form representation, then 
it is  
\[
  \alpha(1,2,1)_b,
\]
for some base $b$ and some multiplier $\alpha$.  Hence,
\[
2^n= \alpha (1+2b+b^2)= \alpha (1+b)^2.
\]
Thus $\alpha = 2^i$ and $(1+b)^2=2^j$, where
$i+j =n$ and $2^{i+1} < b=2^{j/2}-1$. 
Thus $i+1 < j/2 =(n-i)/2$.
Hence $3i+2 < n$ and $n\equiv i \pmod{2}$.
\end{proof}
For a given base $b$, is easy to check when the conditions of Theorem \ref{k=2} are satisfied. Theorem~\ref{k=2 BNF} identifies the cases
when they have binomial form.
There are other palindromic representations as well; 
the smallest example is
$2^{12} = (11,6,11)_{19}$. We list all such palindromic representations, for $n \leq 20$, in Table \ref{k=2tab}.

We observe from Table \ref{k=2tab} that there are some palindromic representations having the form $2^{n} = (1,c,1)_{b}$. It is perhaps of interest to consider these representations in more detail.

\begin{theorem}
\label{c=1.thm}
There is a representation $2^{n} = (1,c,1)_{b}$ that is not of binomial form if and only if 
\begin{enumerate}
\item there is a factorization $2^n-1 = kb$ with $b \leq k \leq 2b-1$, and 
\item it is not the case that $n$ is even and $b = 2^{n/2} - 1$.
\end{enumerate}
\end{theorem}
\begin{proof}
Suppose there is a representation $2^{n} = (1,c,1)_{b}$.
From Theorem \ref{k=2} and its proof, we have $2^n \equiv 1 \bmod b$ and $2^n-1 = kb$. Therefore, $\left\lfloor \frac{2^n}{b} \right\rfloor = k$.
Then the inequality (\ref{3digit.eq}) is
\[ \frac{k}{b}  - \frac{b-1}{b} \leq 1 \leq \frac{k}{b},\]
which simplifies to 
\[ k-b+1 \leq b \leq k,\]
or, equivalently,
\[b \leq k \leq 2b-1,\]
where $c = 2^n \bmod b \neq 0$.
Hence there is a factorization $2^n-1 = kb$ with $b \leq k \leq 2b-1$.

We need to consider the possibility that this representation has binomial form. From Theorem \ref{k=2 BNF} with $i=0$, we see that
$b = 2^{n/2}-1$, so $n$ is even and $2^n = (1,2,1)_b$.
\end{proof}

\begin{example} Suppose $n = 15$. We have the prime factorization
$2^{15-1} = 7 \times 31 \times 151$. If we take $b = 151$ and $k = 7 \times 31 = 217$, then the conditions of Theorem \ref{c=1.thm} are satisfied. We obtain
the palindromic representation $2^{15} = (1,66,1)_{151}$. 
\end{example}


\begin{table}
\begin{center}
\caption{Non-binomial palindromic representations with 
three digits for $n \leq 20$\\
$2^n = (c,d,c)_b$}
\label{k=2tab}
\[
\begin{array}{rrrr}
\multicolumn{1}{c}{n}&
\multicolumn{1}{c}{b}&
\multicolumn{1}{c}{c}&
\multicolumn{1}{c}{d}\\
\hline
12 & 19 & 11 & 6\\
13 & 27  &  11 &  6\\
14  & 27 & 22 & 12\\
14 & 60 & 4 & 33\\
15 & 37 & 23 & 34\\
15 & 151 &  1 & 66\\
16 & 151 & 2 & 132\\
17 & 142 & 6 & 71\\
18 & 399 & 1 & 258\\
19 & 269 & 7 & 66\\
19 & 438 & 2 & 321\\
20 & 269 & 14 & 132\\
20 & 775 & 1 & 578\\
20 & 825 & 1 & 446
\end{array}
\]
\end{center}
\end{table}

We now  obtain some numerical conditions that guarantee that a palindromic representation 
has binomial form. 

\vbox{
\begin{lemma}
\label{L2.3}
Suppose $b =2^x-1$ for some $x$
and let $2^n =(c_k,c_{k-1},\ldots,c_0)_b$ be a palindromic representation.
Then the following hold.
\begin{enumerate}[label={\textup{(\arabic *)}}]
\item $n > k(x-1)$.
\item Suppose $r=n-kx \geq 0$. If $k \leq x-r$, or if $3 \leq k \leq x-r+1$ and $x \geq 3$, 
then the palindromic representation of $2^n$ in the base 
$b$ is 
the binomial form given in 
Lemma~\ref{binom}, with $\alpha=2^r$.
\end{enumerate}
\end{lemma}
}
\begin{proof}%
We are assuming that  $b=2^x-1$ for some $x$ 
and  $2^n =(c_k,c_{k-1},\ldots,c_0)_b$ is a palindromic representation.
Hence
\[
2^n = \sum_{i=0}^k c_i b^i
\]
and $c_i= c_{k-i}$, for each $i$. In particular, because $c_0 = c_k \geq 1$, we have 
$2^n \geq b^k + 1 > (2^x-1)^k$. However, $2^x-1 > 2^{x-1}$ because $x \geq 2$. Hence $2^n > (2^{x-1})^k$ and therefore $n > (x-1)k$, which proves (1). 


Now assume that $r=n-kx \geq 0$ and $k \leq x-r$. 
We have 
$
2^n = 2^r(b+1)^k.
$
It is well known (see~\cite{Luke}) that
\[
\binom{2m}{m} =
\frac{4^m}{\sqrt{\pi m}}\left(
1-\frac{1}{8m}+\frac{1}{128m^2}+\frac{5}{1024m^3}+
O(m^{-4})
\right) < \frac{4^m}{\sqrt{\pi m}}
\]
and
\[
\binom{2m-1}{m-1} 
=\frac{1}{2}\binom{2m}{m} < 
\frac{4^{m-\frac{1}{2}}}{\sqrt{\pi m}}.
\]

Thus, when $k\geq 2$ is even:
\[ \binom{k}{\ceil{k/2}} \: = {\dbinom{k}{k/2}}
\: < \: \frac{2^k}{\sqrt{{k\pi}/{2}}}
\: \leq \: \frac{2^k}{\sqrt{\pi}}
\: < \: \frac{2^k}{1.75}
\: \leq \: \frac{2^{x-r}}{1.75},
\]
since $\sqrt{\pi} > 1.75$ and $k \leq x-r$.
Similarly, when $k\geq 1$ is odd:
\[\binom{k}{\ceil{k/2}} \: = \: {\dbinom{k}{(k+1)/2}} \: = \: {\dbinom{k}{(k-1)/2}}
\:  < \:
  \dfrac{ 2^k }{ \sqrt{{(k+1)\pi}/{2}} }
\: \leq \: \frac{2^k}{\sqrt{\pi}}
\: < \: \frac{2^k}{1.75}
\: \leq \: \frac{2^{x-r}}{1.75}.\]
Consequently
\[
2^r \binom{k}{\ceil{k/2}} < \frac{2^{x}}{1.75} < 2^x - 1 = b
,\]
since $x \geq 2$.
Hence, Lemma~\ref{binom} applies with $\alpha=2^r$.

\medskip

The proof is similar when $3 \leq k \leq x-r+1$.
When $k\geq 4$ is even:
\[ \binom{k}{\ceil{k/2}} \: 
\: < \: \frac{2^k}{\sqrt{{k\pi}/{2}}}
\: \leq \: \frac{2^k}{\sqrt{2\pi}}
\: < \: \frac{2^k}{2.5}
\: \leq \: \frac{2^{x-r+1}}{2.5} \: = \: \frac{2^{x-r}}{1.25},
\]
since $\sqrt{2\pi} > 2.5$ and $k \leq x-r+1$.
The same upper bound holds when $k\geq 5$ is odd.

Consequently
\[
2^r \binom{k}{\ceil{k/2}} < \frac{2^{x}}{1.25} < 2^x - 1 = b
,\]
since $x \geq 3$.
Hence, Lemma~\ref{binom} applies with $\alpha=2^r$.
\end{proof}

We can also use the proof technique of Lemma \ref{L2.3} to show the following.

\begin{theorem} For all integers $n \geq 2$, $2^n$ has a base-$b$ binomial form representation for some $b \leq 2^y -1$, where $y \leq \left\lceil \sqrt{2n} \right\rceil + 1$.
\end{theorem} 

\begin{proof}
Define $x = \lceil \sqrt{2n} \rceil$ and  $k = \left\lfloor \frac{n}{x} \right\rfloor$. 
Then $n = xk + r$, where $0 \leq r \leq x-1$. If $r \leq x - k$, then we have a binomial form representation of $2^n$ to the base $b$, where $b \leq 2^x -1$, as shown in Lemma \ref{L2.3}.

Thus we can suppose that $r \geq x - k + 1$. We have $n = (x+1)k + r-k$, so $n = yk + r'$, where $y = x+1$ and $r' = r-k$. 
Since $r \leq x -1$, we have $r' \leq x - k - 1 = y - k - 2 < y-k$. So we will obtain a binomial form representation of $2^n$ to the base $b$, where $b \leq 2^y -1$, provided that $r' \geq 0$, 
i.e., if $r \geq k$. We have $r \geq x - k + 1$, so it is sufficient to show that $x - k + 1 \geq k$, i.e., $x \geq 2k-1$. 

We have $x \geq  \sqrt{2n}$. Since $k \leq n/x$, we have $kx \leq n$. Hence $x \geq  \sqrt{2n} \geq  \sqrt{2kx}$.
Squaring, we obtain $x^2 \geq 2kx$ and hence $x \geq 2k > 2k-1$, as desired. 
\end{proof}

In Table~\ref{powers of 2}, the  minimum-base palindromic representation of $2^n$ 
is computed for all $n \leq 64$. In every case, the minimum base is of the form $2^x-1$ for some integer $x$. 
The hypotheses of Lemma~\ref{L2.3} are satisfied 
satisfied for each such $n$ except $n=63$. Thus, for all $n \neq 63$, the minimum-base palindromic representation of $2^n$ is guaranteed to be of binomial form. However, even for $n=63$, the hypotheses of Lemma~\ref{binom} are
satisfied.

\begin{table}[tb]
\begin{center}
\caption
{Palindromic representations of $2^n$, where $n=kx+r$\\
Lemma~\ref{L2.3} applies to each entry except for $n=63$\\
}\label{powers of 2}
\vspace*{-\baselineskip}
{\small
\begin{tabular}{@{}*{4}{r@{\hspace{2pt}}}@{\hspace{4pt}}l@{\hspace{1pt}}c@{\hspace{1pt}}l@{\hspace{1pt}}l@{\hspace{1pt}}c@{\hspace{1pt}}l@{}}
\hline
$n$&$k$&$x$&$r$&$b(2^n)$&$=$&$2^x{-}1$&\\\hline
 1& 0& 2& 1&$b(2^{1})$&$=$&$3$ & $2^{1}$&$=$&$2\cdot(1)_{3}$\\
 2& 1& 2& 0&$b(2^{2})$&$=$&$3$ & $2^{2}$&$=$&$(1,1)_{3}$\\
 3& 1& 2& 1&$b(2^{3})$&$=$&$3$ & $2^{3}$&$=$&$2\cdot(1,1)_{3}$\\
 4& 2& 2& 0&$b(2^{4})$&$=$&$3$ & $2^{4}$&$=$&$(1,2,1)_{3}$\\
 5& 1& 3& 2&$b(2^{5})$&$=$&$7$ & $2^{5}$&$=$&$4\cdot(1,1)_{7}$\\
 6& 2& 3& 0&$b(2^{6})$&$=$&$7$ & $2^{6}$&$=$&$(1,2,1)_{7}$\\
 7& 2& 3& 1&$b(2^{7})$&$=$&$7$ & $2^{7}$&$=$&$2\cdot(1,2,1)_{7}$\\
 8& 2& 4& 0&$b(2^{8})$&$=$&$15$ & $2^{8}$&$=$&$(1,2,1)_{15}$\\
 9& 3& 3& 0&$b(2^{9})$&$=$&$7$ & $2^{9}$&$=$&$(1,3,3,1)_{7}$\\
10& 3& 3& 1&$b(2^{10})$&$=$&$7$ & $2^{10}$&$=$&$2\cdot(1,3,3,1)_{7}$\\
11& 2& 5& 1&$b(2^{11})$&$=$&$31$ & $2^{11}$&$=$&$2\cdot(1,2,1)_{31}$\\
12& 4& 3& 0&$b(2^{12})$&$=$&$7$ & $2^{12}$&$=$&$(1,4,6,4,1)_{7}$\\
13& 3& 4& 1&$b(2^{13})$&$=$&$15$ & $2^{13}$&$=$&$2\cdot(1,3,3,1)_{15}$\\
14& 3& 4& 2&$b(2^{14})$&$=$&$15$ & $2^{14}$&$=$&$4\cdot(1,3,3,1)_{15}$\\
15& 3& 5& 0&$b(2^{15})$&$=$&$31$ & $2^{15}$&$=$&$(1,3,3,1)_{31}$\\
16& 4& 4& 0&$b(2^{16})$&$=$&$15$ & $2^{16}$&$=$&$(1,4,6,4,1)_{15}$\\
17& 4& 4& 1&$b(2^{17})$&$=$&$15$ & $2^{17}$&$=$&$2\cdot(1,4,6,4,1)_{15}$\\
18& 3& 5& 3&$b(2^{18})$&$=$&$31$ & $2^{18}$&$=$&$8\cdot(1,3,3,1)_{31}$\\
19& 3& 6& 1&$b(2^{19})$&$=$&$63$ & $2^{19}$&$=$&$2\cdot(1,3,3,1)_{63}$\\
20& 5& 4& 0&$b(2^{20})$&$=$&$15$ & $2^{20}$&$=$&$(1,5,10,10,5,1)_{15}$\\
21& 4& 5& 1&$b(2^{21})$&$=$&$31$ & $2^{21}$&$=$&$2\cdot(1,4,6,4,1)_{31}$\\
22& 4& 5& 2&$b(2^{22})$&$=$&$31$ & $2^{22}$&$=$&$4\cdot(1,4,6,4,1)_{31}$\\
23& 3& 7& 2&$b(2^{23})$&$=$&$127$ & $2^{23}$&$=$&$4\cdot(1,3,3,1)_{127}$\\
24& 4& 6& 0&$b(2^{24})$&$=$&$63$ & $2^{24}$&$=$&$(1,4,6,4,1)_{63}$\\
25& 5& 5& 0&$b(2^{25})$&$=$&$31$ & $2^{25}$&$=$&$(1,5,10,10,5,1)_{31}$\\
26& 5& 5& 1&$b(2^{26})$&$=$&$31$ & $2^{26}$&$=$&$2\cdot(1,5,10,10,5,1)_{31}$\\
27& 4& 6& 3&$b(2^{27})$&$=$&$63$ & $2^{27}$&$=$&$8\cdot(1,4,6,4,1)_{63}$\\
28& 4& 7& 0&$b(2^{28})$&$=$&$127$ & $2^{28}$&$=$&$(1,4,6,4,1)_{127}$\\
29& 4& 7& 1&$b(2^{29})$&$=$&$127$ & $2^{29}$&$=$&$2\cdot(1,4,6,4,1)_{127}$\\
30& 6& 5& 0&$b(2^{30})$&$=$&$31$ & $2^{30}$&$=$&$(1,6,15,20,15,6,1)_{31}$\\
31& 5& 6& 1&$b(2^{31})$&$=$&$63$ & $2^{31}$&$=$&$2\cdot(1,5,10,10,5,1)_{63}$\\
32& 5& 6& 2&$b(2^{32})$&$=$&$63$ & $2^{32}$&$=$&$4\cdot(1,5,10,10,5,1)_{63}$\\
\hline
\end{tabular}
}
{\small
\begin{tabular}{@{}*{4}{r@{\hspace{2pt}}}@{\hspace{4pt}}l@{\hspace{1pt}}c@{\hspace{1pt}}l@{\hspace{1pt}}l@{\hspace{1pt}}c@{\hspace{1pt}}l@{}}
\hline
$n$&$k$&$x$&$r$&$b(2^n)$&$=$&$2^x{-}1$&\\\hline
33& 4& 8& 1&$b(2^{33})$&$=$&$255$ & $2^{33}$&$=$&$2\cdot(1,4,6,4,1)_{255}$\\
34& 4& 8& 2&$b(2^{34})$&$=$&$255$ & $2^{34}$&$=$&$4\cdot(1,4,6,4,1)_{255}$\\
35& 5& 7& 0&$b(2^{35})$&$=$&$127$ & $2^{35}$&$=$&$(1,5,10,10,5,1)_{127}$\\
36& 6& 6& 0&$b(2^{36})$&$=$&$63$ & $2^{36}$&$=$&$(1,6,15,20,15,6,1)_{63}$\\
37& 6& 6& 1&$b(2^{37})$&$=$&$63$ & $2^{37}$&$=$&$2\cdot(1,6,15,20,15,6,1)_{63}$\\
38& 5& 7& 3&$b(2^{38})$&$=$&$127$ & $2^{38}$&$=$&$8\cdot(1,5,10,10,5,1)_{127}$\\
39& 4& 9& 3&$b(2^{39})$&$=$&$511$ & $2^{39}$&$=$&$8\cdot(1,4,6,4,1)_{511}$\\
40& 5& 8& 0&$b(2^{40})$&$=$&$255$ & $2^{40}$&$=$&$(1,5,10,10,5,1)_{255}$\\
41& 5& 8& 1&$b(2^{41})$&$=$&$255$ & $2^{41}$&$=$&$2\cdot(1,5,10,10,5,1)_{255}$\\
42& 7& 6& 0&$b(2^{42})$&$=$&$63$ & $2^{42}$&$=$&$(1,7,21,35,35,21,7,1)_{63}$\\
43& 6& 7& 1&$b(2^{43})$&$=$&$127$ & $2^{43}$&$=$&$2\cdot(1,6,15,20,15,6,1)_{127}$\\
44& 6& 7& 2&$b(2^{44})$&$=$&$127$ & $2^{44}$&$=$&$4\cdot(1,6,15,20,15,6,1)_{127}$\\
45& 5& 9& 0&$b(2^{45})$&$=$&$511$ & $2^{45}$&$=$&$(1,5,10,10,5,1)_{511}$\\
46& 5& 9& 1&$b(2^{46})$&$=$&$511$ & $2^{46}$&$=$&$2\cdot(1,5,10,10,5,1)_{511}$\\
47& 5& 9& 2&$b(2^{47})$&$=$&$511$ & $2^{47}$&$=$&$4\cdot(1,5,10,10,5,1)_{511}$\\
48& 6& 8& 0&$b(2^{48})$&$=$&$255$ & $2^{48}$&$=$&$(1,6,15,20,15,6,1)_{255}$\\
49& 7& 7& 0&$b(2^{49})$&$=$&$127$ & $2^{49}$&$=$&$(1,7,21,35,35,21,7,1)_{127}$\\
50& 7& 7& 1&$b(2^{50})$&$=$&$127$ & $2^{50}$&$=$&$2\cdot(1,7,21,35,35,21,7,1)_{127}$\\
51& 6& 8& 3&$b(2^{51})$&$=$&$255$ & $2^{51}$&$=$&$8\cdot(1,6,15,20,15,6,1)_{255}$\\
52& 5&10& 2&$b(2^{52})$&$=$&$1023$ & $2^{52}$&$=$&$4\cdot(1,5,10,10,5,1)_{1023}$\\
53& 5&10& 3&$b(2^{53})$&$=$&$1023$ & $2^{53}$&$=$&$8\cdot(1,5,10,10,5,1)_{1023}$\\
54& 6& 9& 0&$b(2^{54})$&$=$&$511$ & $2^{54}$&$=$&$(1,6,15,20,15,6,1)_{511}$\\
55& 6& 9& 1&$b(2^{55})$&$=$&$511$ & $2^{55}$&$=$&$2\cdot(1,6,15,20,15,6,1)_{511}$\\
56& 8& 7& 0&$b(2^{56})$&$=$&$127$ & $2^{56}$&$=$&$(1,8,28,56,70,56,28,8,1)_{127}$\\
57& 7& 8& 1&$b(2^{57})$&$=$&$255$ & $2^{57}$&$=$&$2\cdot(1,7,21,35,35,21,7,1)_{255}$\\
58& 7& 8& 2&$b(2^{58})$&$=$&$255$ & $2^{58}$&$=$&$4\cdot(1,7,21,35,35,21,7,1)_{255}$\\
59& 5&11& 4&$b(2^{59})$&$=$&$2047$ & $2^{59}$&$=$&$16\cdot(1,5,10,10,5,1)_{2047}$\\
60& 6&10& 0&$b(2^{60})$&$=$&$1023$ & $2^{60}$&$=$&$(1,6,15,20,15,6,1)_{1023}$\\
61& 6&10& 1&$b(2^{61})$&$=$&$1023$ & $2^{61}$&$=$&$2\cdot(1,6,15,20,15,6,1)_{1023}$\\
62& 6&10& 2&$b(2^{62})$&$=$&$1023$ & $2^{62}$&$=$&$4\cdot(1,6,15,20,15,6,1)_{1023}$\\
\R{63}&\R{ 9}&\R{ 7}&\R{ 0}&$b(2^{63})$&$=$&$127$ & $2^{63}$&$=$&$(1,9,36,84,126,126,84,36,9,1)_{127}$\\
64& 8& 8& 0&$b(2^{64})$&$=$&$255$ & $2^{64}$&$=$&$(1,8,28,56,70,56,28,8,1)_{255}$\\
\hline
\end{tabular}
}
\end{center}
\end{table}

We also have computed all palindromic representations of $2^n$ 
for every positive integer $n < 64.$
Perhaps surprisingly, these computations show that every palindromic representation 
of $2^n$ is either of binomial form or has three digits.

\section{Some Extensions}

\begin{theorem}\label{Leonard-p}
If $p$ is a prime and  $N=p^n$ has a palindromic representation
$N=(c_k, c_{k-1},\ldots,c_0)_b$ with $k$ odd,
then $b=p^x-1$ for some $x$.
\end{theorem}
\begin{proof}
Applying Lemma~\ref{Marco}, we see that $b+1$ divides $p^n$.
Therefore $b=p^x-1$, for some $x$.
\end{proof}

An obvious theorem whose proof we leave for the reader is Theorem~\ref{Easy}.

\begin{theorem}\label{Easy}
For any positive integer $z$.
\begin{enumerate}[label={\textup{(\Roman*)}}]
\item $z^n$ has the $(n+1)$-digit representation $z = (1,0,0,...,0)_z$.
\item $z^n+1$ has the $(n+1)$-digit palindromic representation $z = (1,0,0,...,1)_z$.
\item $z^n-1$ has the $n$-digit palindromic representation $z = (1,1,1,...,1)_z$.
\end{enumerate}
\end{theorem}
\medskip
\noindent
Note that Theorem~\ref{Easy} implies $b(2^n\pm 1)=2$.

\medskip

A solution $(x,y,n,q)$
to the Nagell-Ljunggren Diophantine equation
\begin{equation}\label{NL}
\frac{x^n-1}{x-1} = y^q
\end{equation}
is equivalent to the $n$-digit palindromic representation 
\[ 
y^q = (1,1,1,...,1)_z.
\]
Bugeaud and Mih\u{a}ilescu studied the Nagell-Ljunggren Diophantine equation
and obtained the following theorem.
\begin{theorem}[Bugeaud and Mih\u{a}ilescu~\cite{BugeaudMihailescu}]
Apart from the solutions
\[
  11^2 = (1,1,1,1,1)_3,\qquad
  20^2 = (1,1,1,1)_7\quad
  \text{and}\quad
  7^3 = (1,1,1)_{18},
\]
equation~\textup{(\ref{NL})} has no other solution $(x,y,n,q)$ if any of 
the following conditions are satisfied:
\begin{enumerate}[label={\textup{(\roman*)}}]
\item $q=2$,
\item $3$ divides $n$,
\item $4$ divides $n$,
\item $q=3$ and $n\not\equiv 5 \pmod{6}.$
\end{enumerate}
\end{theorem}

We conclude this section with some computational results that are presented in  Tables \ref{tab5} and \ref{tab4} in the Appendix.

\section{Concluding remarks}
We have several conjectures:
\begin{enumerate}[label={(\alph *)}]
\item
$b(2^n) = 2^x-1$, for some $x$.
\item
The minimum palindromic representation of $2^n$ has 
binomial form.
\item $b(2^{a^2}) = 2^a-1$.
\item For any base $b$ there are only finitely many integers $N=2^n$
such that $b(N)=b$.
\item $b(2^n)=3$ if and only if $n=1,2,3$ or $4$.
\end{enumerate}

\section{Acknowledgements}
Every year on his birthday, the first author reports to Marco Buratti of Rome some amusing palindromic connection. (Marco is obsessed with palindromes.) This past year (2023) he reported that
he turned 68 which is a palindrome in base 3 and also in base 16.
The interlocution quickly turned to a discussion on what is 
the smallest base  such that N is a palindrome in that base.
Data was generated and on 22 September 2023, Marco wrote:
\begin{quotation}
\noindent%
Dear Don,\\
thank you for your interesting message!\\
Is $b(2^n)$ always of the form $2^i-1$?\\
Maybe my question is stupid.\\
Ciao,\\
Marco
\end{quotation}
It turns out that it is not so stupid after all.

\pagebreak

\appendix

\section*{Appendix}

\begin{table}[htbp]
\begin{center} 
\caption%
{Representations of $2^n$ in the base 3 for $n \leq 25$.\\
Non-palindromic forms are marked in red.}
\label{tab5}
\begin{tabular}{@{}l@{\hspace{2pt}}c@{\hspace{2pt}}l@{}}
\hline
$2^{1}$&$=$&$2\cdot(1)_{3}$\\
$2^{2}$&$=$&$(1,1)_{3}$\\
$2^{3}$&$=$&$2\cdot(1,1)_{3}$\\
$2^{4}$&$=$&$(1,2,1)_{3}$\\
$2^{5}$&$=$&\R{$(1,0,1,2)_{3}$}\\
$2^{6}$&$=$&\R{$(2,1,0,1)_{3}$}\\
$2^{7}$&$=$&\R{$(1,1,2,0,2)_{3}$}\\
$2^{8}$&$=$&\R{$(1,0,0,1,1,1)_{3}$}\\
$2^{9}$&$=$&\R{$2\cdot(1,0,0,1,1,1)_{3}$}\\
$2^{10}$&$=$&\R{$(1,1,0,1,2,2,1)_{3}$}\\
$2^{11}$&$=$&\R{$(2,2,1,0,2,1,2)_{3}$}\\
$2^{12}$&$=$&\R{$(1,2,1,2,1,2,0,1)_{3}$}\\
$2^{13}$&$=$&\R{$(1,0,2,0,2,0,1,0,2)_{3}$}\\
$2^{14}$&$=$&\R{$(2,1,1,1,1,0,2,1,1)_{3}$}\\
$2^{15}$&$=$&\R{$(1,1,2,2,2,2,1,1,2,2)_{3}$}\\
$2^{16}$&$=$&\R{$(1,0,0,2,2,2,2,0,0,2,1)_{3}$}\\
$2^{17}$&$=$&\R{$(2,0,1,2,2,2,1,0,1,1,2)_{3}$}\\
$2^{18}$&$=$&\R{$(1,1,1,0,2,2,1,2,1,0,0,1)_{3}$}\\
$2^{19}$&$=$&\R{$(2,2,2,1,2,2,0,1,2,0,0,2)_{3}$}\\
$2^{20}$&$=$&\R{$(1,2,2,2,0,2,1,1,0,1,0,1,1)_{3}$}\\
$2^{21}$&$=$&\R{$(1,0,2,2,1,1,1,2,2,0,2,0,2,2)_{3}$}\\
$2^{22}$&$=$&\R{$(2,1,2,2,0,0,0,2,1,1,1,1,2,1)_{3}$}\\
$2^{23}$&$=$&\R{$(1,2,0,2,1,0,0,1,2,0,0,0,0,1,2)_{3}$}\\
$2^{24}$&$=$&\R{$(1,0,1,1,1,2,0,1,0,1,0,0,0,1,0,1)_{3}$}\\
$2^{25}$&$=$&\R{$(2,1,0,0,0,1,0,2,0,2,0,0,0,2,0,2)_{3}$}\\
$2^{26}$&$=$&\R{$(1,1,2,0,0,0,2,1,1,1,1,0,0,1,1,1,1)_{3}$}\\
$2^{27}$&$=$&\R{$(1,0,0,1,0,0,1,1,2,2,2,2,0,0,2,2,2,2)_{3}$}\\
$2^{28}$&$=$&\R{$(2,0,0,2,0,1,0,0,2,2,2,1,0,1,2,2,2,1)_{3}$}\\
$2^{29}$&$=$&\R{$(1,1,0,1,1,0,2,0,1,2,2,1,2,1,0,2,2,1,2)_{3}$}\\
$2^{30}$&$=$&\R{$(2,2,0,2,2,1,1,1,0,2,2,0,1,2,1,2,2,0,1)_{3}$}\\
\hline
\end{tabular}
\end{center}
\end{table}

\begin{table}[htbp]
\begin{center} 
\caption%
{Palindromic representations of $p^n < 2^{30}$, where $3 \leq p \leq 29$ is prime.\\
Non-binomial forms are marked in red.}
\label{tab4}
\begin{tabular}{
 @{}
 l@{\hspace{2pt}}
 c@{\hspace{2pt}}
 r
 l@{\hspace{2pt}}
 c@{\hspace{2pt}}
 l
 @{}
}
\hline
$b(3^{1})$&$=$&$2$&$3^{1}$&$=$&$(1,1)_{2}$\\
$b(3^{2})$&$=$&$2$&$3^{2}$&$=$&\R{$(1,0,0,1)_{2}$}\\
$b(3^{3})$&$=$&$2$&$3^{3}$&$=$&\R{$(1,1,0,1,1)_{2}$}\\
$b(3^{4})$&$=$&$8$&$3^{4}$&$=$&$(1,2,1)_{8}$\\
$b(3^{5})$&$=$&$8$&$3^{5}$&$=$&$3\cdot(1,2,1)_{8}$\\
$b(3^{6})$&$=$&$8$&$3^{6}$&$=$&$(1,3,3,1)_{8}$\\
$b(3^{7})$&$=$&$24$&$3^{7}$&$=$&\R{$(3,19,3)_{24}$}\\
$b(3^{8})$&$=$&$8$&$3^{8}$&$=$&$(1,4,6,4,1)_{8}$\\
$b(3^{9})$&$=$&$26$&$3^{9}$&$=$&$(1,3,3,1)_{26}$\\
$b(3^{10})$&$=$&$26$&$3^{10}$&$=$&$3\cdot(1,3,3,1)_{26}$\\
$b(3^{11})$&$=$&$80$&$3^{11}$&$=$&$27\cdot(1,2,1)_{80}$\\
$b(3^{12})$&$=$&$26$&$3^{12}$&$=$&$(1,4,6,4,1)_{26}$\\
$b(3^{13})$&$=$&$26$&$3^{13}$&$=$&$3\cdot(1,4,6,4,1)_{26}$\\
$b(3^{14})$&$=$&$80$&$3^{14}$&$=$&$9\cdot(1,3,3,1)_{80}$\\
$b(3^{15})$&$=$&$26$&$3^{15}$&$=$&$(1,5,10,10,5,1)_{26}$\\
$b(3^{16})$&$=$&$80$&$3^{16}$&$=$&$(1,4,6,4,1)_{80}$\\
$b(3^{17})$&$=$&$80$&$3^{17}$&$=$&$3\cdot(1,4,6,4,1)_{80}$\\
$b(3^{18})$&$=$&$26$&$3^{18}$&$=$&$(1,6,15,20,15,6,1)_{26}$\\
$b(5^{1})$&$=$&$2$&$5^{1}$&$=$&\R{$(1,0,1)_{2}$}\\
$b(5^{2})$&$=$&$4$&$5^{2}$&$=$&$(1,2,1)_{4}$\\
$b(5^{3})$&$=$&$4$&$5^{3}$&$=$&$(1,3,3,1)_{4}$\\
$b(5^{4})$&$=$&$24$&$5^{4}$&$=$&$(1,2,1)_{24}$\\
$b(5^{5})$&$=$&$24$&$5^{5}$&$=$&$5\cdot(1,2,1)_{24}$\\
$b(5^{6})$&$=$&$24$&$5^{6}$&$=$&$(1,3,3,1)_{24}$\\
$b(5^{7})$&$=$&$24$&$5^{7}$&$=$&$5\cdot(1,3,3,1)_{24}$\\
$b(5^{8})$&$=$&$24$&$5^{8}$&$=$&$(1,4,6,4,1)_{24}$\\
$b(5^{9})$&$=$&$124$&$5^{9}$&$=$&$(1,3,3,1)_{124}$\\
$b(5^{10})$&$=$&$24$&$5^{10}$&$=$&$(1,5,10,10,5,1)_{24}$\\
$b(5^{11})$&$=$&$124$&$5^{11}$&$=$&$25\cdot(1,3,3,1)_{124}$\\
$b(5^{12})$&$=$&$24$&$5^{12}$&$=$&$(1,6,15,20,15,6,1)_{24}$\\
$b(7^{1})$&$=$&$2$&$7^{1}$&$=$&\R{$(1,1,1)_{2}$}\\
$b(7^{2})$&$=$&$6$&$7^{2}$&$=$&$(1,2,1)_{6}$\\
$b(7^{3})$&$=$&$6$&$7^{3}$&$=$&$(1,3,3,1)_{6}$\\
$b(7^{4})$&$=$&$18$&$7^{4}$&$=$&\R{$7\cdot(1,1,1)_{18}$}\\
$b(7^{5})$&$=$&$38$&$7^{5}$&$=$&\R{$(11,24,11)_{38}$}\\
$b(7^{6})$&$=$&$18$&$7^{6}$&$=$&\R{$(1,2,3,2,1)_{18}$}\\
$b(7^{7})$&$=$&$48$&$7^{7}$&$=$&$7\cdot(1,3,3,1)_{48}$\\
$b(7^{8})$&$=$&$48$&$7^{8}$&$=$&$(1,4,6,4,1)_{48}$\\
$b(7^{9})$&$=$&$18$&$7^{9}$&$=$&\R{$(1,3,6,7,6,3,1)_{18}$}\\
$b(7^{10})$&$=$&$48$&$7^{10}$&$=$&$(1,5,10,10,5,1)_{48}$\\
$b(11^{1})$&$=$&$10$&$11^{1}$&$=$&$(1,1)_{10}$\\
\hline
\end{tabular}
\begin{tabular}{
 @{}
 l@{\hspace{2pt}}
 c@{\hspace{2pt}}
 r
 l@{\hspace{2pt}}
 c@{\hspace{2pt}}
 l
 @{}
}
\hline
$b(11^{2})$&$=$&$3$&$11^{2}$&$=$&\R{$(1,1,1,1,1)_{3}$}\\
$b(11^{3})$&$=$&$10$&$11^{3}$&$=$&$(1,3,3,1)_{10}$\\
$b(11^{4})$&$=$&$10$&$11^{4}$&$=$&$(1,4,6,4,1)_{10}$\\
$b(11^{5})$&$=$&$56$&$11^{5}$&$=$&\R{$(51,19,51)_{56}$}\\
$b(11^{6})$&$=$&$35$&$11^{6}$&$=$&\R{$(1,6,11,6,1)_{35}$}\\
$b(11^{7})$&$=$&$120$&$11^{7}$&$=$&$11\cdot(1,3,3,1)_{120}$\\
$b(11^{8})$&$=$&$120$&$11^{8}$&$=$&$(1,4,6,4,1)_{120}$\\
$b(13^{1})$&$=$&$3$&$13^{1}$&$=$&\R{$(1,1,1)_{3}$}\\
$b(13^{2})$&$=$&$12$&$13^{2}$&$=$&$(1,2,1)_{12}$\\
$b(13^{3})$&$=$&$12$&$13^{3}$&$=$&$(1,3,3,1)_{12}$\\
$b(13^{4})$&$=$&$12$&$13^{4}$&$=$&$(1,4,6,4,1)_{12}$\\
$b(13^{5})$&$=$&$12$&$13^{5}$&$=$&$(1,5,10,10,5,1)_{12}$\\
$b(13^{6})$&$=$&$168$&$13^{6}$&$=$&$(1,3,3,1)_{168}$\\
$b(13^{7})$&$=$&$168$&$13^{7}$&$=$&$13\cdot(1,3,3,1)_{168}$\\
$b(13^{8})$&$=$&$168$&$13^{8}$&$=$&$(1,4,6,4,1)_{168}$\\
$b(17^{1})$&$=$&$2$&$17^{1}$&$=$&\R{$(1,0,0,0,1)_{2}$}\\
$b(17^{2})$&$=$&$4$&$17^{2}$&$=$&\R{$(1,0,2,0,1)_{4}$}\\
$b(17^{3})$&$=$&$4$&$17^{3}$&$=$&\R{$(1,0,3,0,3,0,1)_{4}$}\\
$b(17^{4})$&$=$&$16$&$17^{4}$&$=$&$(1,4,6,4,1)_{16}$\\
$b(17^{5})$&$=$&$16$&$17^{5}$&$=$&$(1,5,10,10,5,1)_{16}$\\
$b(17^{6})$&$=$&$63$&$17^{6}$&$=$&\R{$(1,33,33,33,1)_{63}$}\\
$b(17^{7})$&$=$&$288$&$17^{7}$&$=$&$17\cdot(1,3,3,1)_{288}$\\
$b(19^{1})$&$=$&$18$&$19^{1}$&$=$&$(1,1)_{18}$\\
$b(19^{2})$&$=$&$15$&$19^{2}$&$=$&\R{$(1,9,1)_{15}$}\\
$b(19^{3})$&$=$&$18$&$19^{3}$&$=$&$(1,3,3,1)_{18}$\\
$b(19^{4})$&$=$&$18$&$19^{4}$&$=$&$(1,4,6,4,1)_{18}$\\
$b(19^{5})$&$=$&$18$&$19^{5}$&$=$&$(1,5,10,10,5,1)_{18}$\\
$b(19^{6})$&$=$&$360$&$19^{6}$&$=$&$(1,3,3,1)_{360}$\\
$b(19^{7})$&$=$&$360$&$19^{7}$&$=$&$19\cdot(1,3,3,1)_{360}$\\
$b(23^{1})$&$=$&$3$&$23^{1}$&$=$&\R{$(2,1,2)_{3}$}\\
$b(23^{2})$&$=$&$22$&$23^{2}$&$=$&$(1,2,1)_{22}$\\
$b(23^{3})$&$=$&$22$&$23^{3}$&$=$&$(1,3,3,1)_{22}$\\
$b(23^{4})$&$=$&$22$&$23^{4}$&$=$&$(1,4,6,4,1)_{22}$\\
$b(23^{5})$&$=$&$22$&$23^{5}$&$=$&$(1,5,10,10,5,1)_{22}$\\
$b(23^{6})$&$=$&$22$&$23^{6}$&$=$&$(1,6,15,20,15,6,1)_{22}$\\
$b(29^{1})$&$=$&$4$&$29^{1}$&$=$&\R{$(1,3,1)_{4}$}\\
$b(29^{2})$&$=$&$21$&$29^{2}$&$=$&\R{$(1,19,1)_{21}$}\\
$b(29^{3})$&$=$&$28$&$29^{3}$&$=$&$(1,3,3,1)_{28}$\\
$b(29^{4})$&$=$&$28$&$29^{4}$&$=$&$(1,4,6,4,1)_{28}$\\
$b(29^{5})$&$=$&$28$&$29^{5}$&$=$&$(1,5,10,10,5,1)_{28}$\\
$b(29^{6})$&$=$&$28$&$29^{6}$&$=$&$(1,6,15,20,15,6,1)_{28}$\\
\hline
\end{tabular}
\end{center}
\end{table}

\end{document}